\documentclass[10pt, reqno]{amsart} % Font size, right-side equation labels, ams article format

\usepackage[utf8]{inputenc} % utf8 encoding
\usepackage{graphicx} % Required for inserting images
\graphicspath{{./images/}}
\usepackage{amsopn,amssymb, amsmath, mathtools, amsthm, mathrsfs,xcolor,bm,url,enumitem} % Extra packages
\usepackage{dsfont}
\usepackage[a4paper,lmargin=3cm,rmargin=3cm,tmargin=4cm,bmargin=4cm,marginparwidth=2.8cm,marginparsep=1mm]{geometry} % Reduce side margins a bit
\usepackage[bookmarks=true,hyperindex,pdftex,colorlinks,citecolor=blue]{hyperref} % Hyperlinks to citations, eqrefs and refs 
\usepackage{cleveref}
\usepackage{datetime} % Custom date format commands
\usepackage{tabularx, ragged2e} % Table fixes
\usepackage[colorinlistoftodos]{todonotes}
\usepackage{thmtools, thm-restate}

\usepackage[backend=biber, giveninits=true]{biblatex}
\usepackage{xcolor}
\usepackage{color}
\usepackage{soul}

\parskip\medskipamount  % Paragraph skip size

\newdateformat{monthyeardate}{\monthname[\THEMONTH], \THEYEAR} % date format command for month, year without date
\newdateformat{monthdayyeardate}{\monthname[\THEMONTH] \THEDAY, \THEYEAR}

\newtheorem{theorem}{Theorem}

\newtheorem{lemma}{Lemma}[section]

\newtheorem{proposition}[lemma]{Proposition}

\theoremstyle{definition}

\numberwithin{equation}{section}

\newcolumntype{C}{>{\Centering\arraybackslash}X} % tabularx centered "X" column

\newcommand{\R}[0]{\mathbb{R}} % reals
\newcommand{\N}[0]{\mathbb{N}} % naturals
\newcommand{\Z}[0]{\mathbb{Z}} % integers
\newcommand{\A}[0]{\mathcal{A}}

\newcommand{\E}[0]{\mathcal{E}}

\newcommand{\ep}{\varepsilon} % Epsilon
\newcommand{\ind}[1]{\ensuremath{\mbox{$\mathds{1}$}_{#1}}} % Indicator
\newcommand{\lint}[4]{\ensuremath{\int_{#1}^{#2}{#3}\:\mathrm{d}{#4}}} % Integral
\newcommand{\map}[3]{\ensuremath{{#1}:{#2}\to{#3}}} % Function f: X -> Y
\newcommand{\me}{\ensuremath{\mathrm{e}}} % Non-italic e, for exponential function
\newcommand{\pn}[2]{\ensuremath{\left\|{#1}\right\|_{#2}}} % p-norm
\newcommand{\set}[2]{\ensuremath{\left\{{#1}\;:\;\,{#2}\right\}}} % Set notation

\DeclareMathOperator{\lcm}{lcm} % Lowest common multiple
\DeclareMathOperator{\supp}{supp}
\DeclareMathOperator{\Log}{Log}

\title{Quantitative results on the \texorpdfstring{$k$}{k}-dimensional Duffin-Schaeffer conjecture}

\author[C. O'Reilly]{Connor O'Reilly}
\address[C. O'Reilly]{Norwegian University of Science and Technology, Department of Mathematical Sciences, 7491 Trondheim, Norway}
\email{connoro@ntnu.no}

\date{\monthyeardate\today} % Month year

\begin{document}

\begin{abstract}
    For all $k\geq 2$, we provide almost-sharp quantitative results for the $k$-dimensional Duffin-Schaeffer conjecture, analogous to recent developments in the 1-D case of Koukoulopoulos-Maynard-Yang. In particular, for $\map{\psi}{\N}{[0,1/2]}$ such that $\sum_{q\in \N}(\psi(q)\varphi(q)/q)^k$ diverges, $Q\geq 1$ and $\alpha\in\R$, we denote by $S_k(\alpha, Q)$ the number of pairs $(a,q)\in\Z^k\times \N$ with $q\leq Q$, $\gcd(a_i,q)=1$ for each $i\in\{1,\dots,k\}$, satisfying $\pn{q\alpha-a}{\infty}<\psi(q)$. Defining $\Psi_k(Q)=\sum_{q\leq Q}(2\psi(q)\varphi(q)/q)^k$, we show that for all $\ep>0$ and almost all $\alpha$ one has $S_k(\alpha,Q)=\Psi_k(Q)+O_{\ep,k}(\Psi(Q)^{1/2+\ep})$.
\end{abstract}

\keywords{Diophantine approximation, metric number theory, Duffin-Schaeffer conjecture, simultaneous approximation}
\subjclass{11J83, 11J13, 11K60}
\maketitle

% \todoCOR[]{To do: fix typographical orphans, widows (these are real terms) etc., and where sections start on pages (some of this is broken by the todo margins..)}
% --------------------
\section{Introduction}
% \todoMH[inline]{I would state both Dirichlet and Khintchine in arbitrary dimension, both is true - since we treat the higher-dimensional one, this makes more sense.}

Simultaneous Diophantine approximation concerns the existence of rational approximations of multidimensional irrational points: a classical result of Dirichlet states that, given $k\geq 1$, for every $\alpha\in \R^k$, there are infinitely many solutions $(a,q)\in\Z^k\times \N$ to the inequality \cite[Chapter~2]{schmidt1980diophantine}:
\begin{align*}
    \pn{q\alpha - a}{\infty}<q^{-1/k}.
\end{align*}
% \todoMH[inline]{Should be $q^{-1/k}$}
Moreover, this bound is optimal up to a constant for numbers with bounded partial quotients when represented as continued fractions; this class of numbers are known as badly approximable, and contains e.g. for $k=1$ the golden ratio $\varphi$, and $\sqrt{2}$. However, for typical $\alpha$, in the sense of Lebesgue measure, it turns out that this bound can be significantly improved: metric Diophantine approximation asks the question of by exactly how much. For an overview of metric Diophantine approximation (and metric number theory, in general), one should read the monograph \cite{harman1998metricnumbertheory}. 

A notable foundational result of metric Diophantine approximation is Khintchine's theorem: given a function $\map{\psi}{\N}{[0,\infty)}$ monotonically decreasing, Khintchine's theorem states that for almost every $\alpha\in \R^k$, there exists infinitely many pairs $(a,q)\in\Z^k\times \N$ satisfying the inequality:
\begin{align*}
    \pn{q\alpha-a}{\infty}< \psi(q),
\end{align*}
if and only if $\sum_{q\in\N}(\psi(q))^k$ diverges \cite[Section~3]{beresnevich2009classical}. Intuitively, one can choose $\psi(q)=1/q\log q$ to see that on average the bound from Dirichlet's theorem can be improved by a logarithmic factor. In the case of only finitely many solutions, Khintchine's theorem is immediate from the convergence Borel-Cantelli lemma. The divergence lemma, however, can not be applied directly as there is no stochastic independence, but when $\psi$ is monotonic, the set system under consideration behaves approximately independent in a sense that is sufficient for the statement to hold.
% \todoMH[inline]{I would rephrase the above a bit: Even monotonicity does not give you ``independence'' in a stochastical sense as it is necessary for the divergence BC. I would suggest something like ``divergence BC can't be applied'' since no stochastical independence. for monotonic $\psi$, the sets under consideration behave ``independently enough'' - this is of course vague, but maybe still fine for an intro. Just as information to you: Actually, the set system even in the monotonic case in dimension one does not satisfy the 2nd moment bound necessary. Khintchine original proof worked differently, and another method is in Khintchine's 1926 higher-dimensional paper where he just considers reduced fractions, where such independence can be found (which is all clearer in the work of Duffin and Schaeffer). This note does not need to be mentioned, but keep that in mind so that you don't state something wrong.}

In 1941, Duffin and Schaeffer \cite{duffin1941khintchinesproblem} showed that Khintchine's theorem in 1-dimension fails without the monotonicity condition, by constructing a counterexample where the series $\sum_{q\in \N} \psi(q)$ diverges, yet, for almost every $\alpha$, only finitely many solutions exist. Their work relies on the multiple representations of fractions: if $p/q$ is a close approximation for $\alpha$, then so is $np/nq$ for all $n\in\N$. As a result, if the function $\psi$ is constructed to have values disproportionately high on integer multiples of $q$, the sum may still diverge, despite the $\alpha$ admitting good approximations being relatively rare.
% \todoMH[inline]{Exploit the independence condition.. I don't understand that, probably you meant something different.}
% \todoCOR[inline]{I've reread the Duffin-Schaeffer 1941 paper and I hope this is better?}
% \todoMH[inline]{Sounds good now!}
This led them to conjecture that a variant of Khintchine's theorem would hold when non-reduced fractions are excluded, as follows: denoting by $\varphi$ the Euler totient function, if $\map{\psi}{\N}{[0,\infty)}$ is a function such that $\sum_{q\in\N}\psi(q)\varphi(q)/q$ diverges, then for almost all $\alpha\in\R$, there exists infinitely many solutions to the inequality:
\begin{align}
    \left\lvert \alpha - \frac{p}{q}\right\rvert < \frac{\psi(q)}{q}.\label{eq:approx-ineq}
\end{align}
The Duffin-Schaeffer conjecture was open for 78 years before it was proven by Koukoulopoulos and Maynard \cite{koukoulopoulos2020duffin}. A full survey on the conjecture and an exposition of its proof can be found in \cite{hauke2025centurymetricdiophantineapproximation}.
% \todoMH[inline]{Since you are the native I doubt myself, but I would have said it's a survey ''on'' the conjecture?}
% \todoCOR[inline]{I believe both are correct (one would do a survey on a topic, but a survey of an area (location), as with a survey of n people, so with an area (topic) it's somewhat ambiguous). But I can use ``on'' if this is more common}
% \todoMH[inline]{Okay good to know!}

The work of \cite{koukoulopoulos2020duffin}, however, left open the question of quantitative results on the number of approximations, assuming some bound on $q$. A quantitative version of Khintchine's theorem was proven by Schmidt in 1960 \cite{schmidt1960ametricaltheorem}; one can also view \cite[Chapter~4]{harman1998metricnumbertheory} for an overview of Schmidt's method. Schmidt-type results on the Duffin-Schaeffer conjecture have been the subject of various recent works \cite{aistleitner2023metrictheoryofapproximations, hauke2024provingduffinschaefferconjecturegcd}, with the sharpest known being of Koukoulopoulos, Maynard and Yang \cite{koukoulopoulos2024sharpquantitativeversionduffinschaeffer}. 
% \todoMH[inline]{Many people in metric d.a. know these quantitative version as "Schmidt-type results", you can call them like that here if you want to.}
% \todoMH[inline]{Maybe some smoother transition here like ''to set the scene'' or whatever so that the reader somehow understand that you now explain what is meant with "quantitative"}
Towards this quantitative question, given $\map{\psi}{\N}{[0,1/2]}$ and $Q\geq 1$, we define:
\begin{align*}
    \Psi(Q) \coloneq \sum_{q\leq Q}\frac{2\psi(q)\varphi(q)}{q}.
\end{align*}
Moreover, denote by $S(\alpha,Q)$ the number of pairs $(p,q)\in \Z \times \N$ with $q\leq Q,\ \gcd(p,q)=1$ that satisfy the inequality \eqref{eq:approx-ineq}. The work of Koukoulopoulos, Maynard and Yang \cite{koukoulopoulos2024sharpquantitativeversionduffinschaeffer} then states that, assuming $\lim_{Q\to\infty}\Psi(Q)=\infty$, for every $\ep>0$ and almost all $\alpha\in\R$, we have as $Q\to \infty$:
\begin{align*}
    S(\alpha,Q)=\Psi(Q)+O_\ep\left(\Psi(Q)^{\frac{1}{2}+\ep}\right).
\end{align*}

% \todoMH[inline]{I would suggest "in this paper we consider" instead, sounds clearer and is more confident. $\checkmark$}
In this paper, we consider a higher-dimensional generalisation of the Duffin-Schaeffer conjecture as follows: for $k\geq 1$, given $\map{\psi}{\N}{[0,\infty)}$ and $\alpha\in \R^k$, we ask whether there exists infinitely many solutions $(a,q)\in\Z^k\times \N$ satisfying the inequality:
\begin{align}
    \pn{q\alpha-a}{\infty}< \psi(q),\label{eq:kdim-approx-ineq}
\end{align}
where $a=(a_1,\dots,a_k)$ satisfies $\gcd(a_i, q)=1$ for $i\in\{1,\dots, k\}$. When $k=1$, this is then exactly the question of the Duffin-Schaeffer conjecture. For $k\geq 2$, the question was first raised by Sprind\v{z}uk \cite{sprindzhuk1979metric}, and then resolved by Pollington and Vaughan \cite{pollington1990kdimensionalduffinschaeffer} as follows: assuming the notation and hypothesis above, there exists infinitely many solutions to the inequality \eqref{eq:kdim-approx-ineq} if and only if the sum $\sum_{q\in\N}(\psi(q)\varphi(q)/q)^k$ diverges.
% \todoMH[inline]{There is of course no real translation of the belarussian name, but I usually see (and use) Sprind\v{z}uk}\todoCOR[inline]{Sure. I've seen it spelled differently in different places; same with Khintchine/Khinchin, so never sure which to take..}
% \todoMH[inline]{That's always with the russian names since there is no unique transcription. But as long as you use it consistently, its fine.}

Following the recent progress in the 1-dimensional case, we now ask the question of quantitative results for the higher-dimensional setup of Pollington and Vaughan. By periodicity, we consider only $\alpha \in [0,1]^k$ and thus $0\leq a \leq q$, and we assume $\psi(q)\leq 1/2$, following \cite{aistleitner2023metrictheoryofapproximations, koukoulopoulos2024sharpquantitativeversionduffinschaeffer}. The main theorem of this paper is the following. We state the result for all $k\geq 1$ for completeness; our contribution, of course, is the case $k\geq 2$.
% \todoMH[inline]{Depending on personal taste, you can state it for $k \geq 1$ - of course it's only new in higher dimension, but nevertheless true.}
\begin{theorem}\label{theorem:k-dim-ds}
   Let $k\geq 1$, and let $\map{\psi}{\N}{\left[0,1/2\right]}$ be such that $\sum_{q=1}^\infty \left(\varphi(q)\psi(q)/q\right)^k=\infty$. For $\alpha\in \left[0,1\right]^k$, denote by $S_k(\alpha,Q)$ the number of solutions $(a,q)\in\Z^k\times \N$ to the inequality:
    \begin{align*}
        \pn{q\alpha-a}{\infty} < \psi(q), \quad \text{given } q \leq Q,
    \end{align*}
    satisfying $\gcd(a_i, q)=1$, for $i\in\{1,\dots, k\}$, $a=(a_1,\dots, a_k)$. Moreover, given $Q\in \N$, define:
    \begin{align*}
        \Psi_k(Q)\coloneq \sum_{q\leq Q}\left(\frac{2\psi(q)\varphi(q)}{q}\right)^k.
    \end{align*}
    Then for all $\ep>0$, as $Q\to \infty$, we have for almost all $\alpha$ that:
    \begin{align*}
        S_k(\alpha,Q)=\Psi_k(Q)+O_{\ep,k}\left( \Psi_k(Q)^{1/2+\ep} \right).
    \end{align*}
\end{theorem}
This result establishes a quantitative estimate for the number of solutions for the $k$-dimensional Duffin-Schaeffer conjecture for all $k\geq 2$, matching the best-known error term in the 1-dimensional case from \cite{koukoulopoulos2024sharpquantitativeversionduffinschaeffer}. We note that, in the original, non-reduced fraction case of Khintchine, without assuming monotonicity, asymptotic formulae exist only in the case $k\geq 3$. For $k=1,2$, we have such formulae with monotonicity; for $k=1$, we know that this monotonicity is essential, but for $k=2$, it remains an open question: see \cite[Chapter~3]{harman1998metricnumbertheory} or \cite[Section~5]{beresnevich2010khintchinegroshev}.
% \todoMH[inline]{I think you can omit (up to the implied constant). Maybe also state that there is no quantitative formula known in the non-reduced case(s), I think this was asked somewhere in a paper of Victor and Sanju (I believe it was "metric d.a. revisited: the Khintchine-Groshev", but I might be wrong).
% By cases, I mean you can only look at joint gcd being $1$, or no gcd condition at all.
% }
% \todoCOR[inline]{Is this correct/what you meant?}
% \todoMH[inline]{Yes this is what I had in mind. I didn't remember correctly though that they speak about the linear form case mainly. 
% They do say something about a certain Schmidt formula not being applicable, but this doesn't mean that there are no formulae. So one clearly has to be more precise what we say above: In dimension $k \geq 3$, we have asymptotic formulae. In dimensions 1 and 2, we have asymptotic formula as soon as we assume monotonicity conditions. The dimension 2 and 3 theory is treated in the GCD sum chapter in Harman's book (chapter 3 I think), so potentially just state Harman chapter 3 and the BV paper as two references, and say that the case $k = 2$ is still open.
% }
% \todoCOR[]{Is this correct? I've tried to read and understand both but not 100 percent getting the BV paper so mostly followed Harman}
% \todoMH[inline]{I would put Theorem 2 also in the intro, I think this is interesting enough on its own.}
% \todoMH[inline]{You can read what we did in the ABH paper between Theorems 1 and Theorem 2, and put a shorted version of this as a transition from Thm 1 to Thm 2 here.}

The method of proving Theorem~\ref{theorem:k-dim-ds} and similar results in metric Diophantine approximation is a standard procedure in metric number theory, with ideas dating back to at least \cite{cassels1950somemetricaltheorems, schmidt1964metricaltheorems}. We omit the proof, referring the reader to \cite[Lemma~1.5]{harman1998metricnumbertheory} for the general procedure and to \cite{koukoulopoulos2024sharpquantitativeversionduffinschaeffer} for the details in our case. The key is a variance estimate of random variables describing the number of solutions for a given $\alpha$. From a probabilistic perspective, the 1-dimensional Duffin-Schaeffer conjecture asks whether a uniformly randomly chosen irrational $\alpha$ is contained within the set system:
\begin{align*}
    \A_q \coloneq \bigcup_{\substack{0\leq a \leq q,\\\gcd(a,q)=1}}\left(\frac{a}{q}-\frac{\psi(q)}{q}, \frac{a}{q} + \frac{\psi(q)}{q}\right)\cap \left[0,1 \right]
\end{align*}
for infinitely or finitely many $q$; for $k\geq2$, we ask about $\A_q^k\coloneq \A_q^{k-1}\times \A_q$. Denoting then by $\lambda_k$ the $k$-dimensional Lebesgue measure restricted to $[0,1]^k$, we note that $\Psi_k(Q) = \sum_{q=1}^Q\lambda_k(\A_q^k)$ as:
\begin{align*}
        \lambda_k(\A_q^k) = \left(\frac{2\psi(q)\varphi(q)}{q}\right)^k.
\end{align*}
Following this, one can also see that $S_k(\alpha,Q)=\sum_{q=1}^Q\ind{\A_q^k}(\alpha)$ for every $\alpha$. The result of Theorem~\ref{theorem:k-dim-ds} can then be understood as describing the concentration of $S_k(\alpha,Q)$ around its mean. As before, the result would follow immediately from the divergence Borel-Cantelli lemma with pairwise independence of the random variables $\ind{\A_q}$, which we do not have here. In fact, the overlap $\lambda_k(\A_q^k\cap \A^k_r)$ can be significantly larger than $\lambda_k(\A_q^k)\lambda_k(\A_r^k)$ for pairs $(q,r)$ where the prime factorisation of $Q$ contains a large number of small primes not dividing $r$, or vice versa. The crux of the argument involves showing that such pairs are relatively rare, and thus on average we have a form of quasi-independence. In fact, one can show that the set system moves towards pairwise independence on average as the total mass tends towards infinity. Our key result below establishes exactly this for the case $k\geq 2$; again, stated as $k\geq 1$ for completeness.

\begin{theorem}\label{theorem:k-dim-variance}
    Let $k\geq 1$ and $Q\in \N$, and define $S_k(\alpha,Q)$ and $\Psi_k(Q)$ as in Theorem~\ref{theorem:k-dim-ds}. Then, for every fixed $\ep>0$, we have:
    \begin{align*}
         \lint{\left[0,1\right]^k}{}{\left(S_k(\alpha, Q)-\Psi_k(Q)\right)^2}{\alpha} \leq \Psi_k(Q) + O_{\ep,k}(\Psi_k(Q)^{1+\ep}).
     \end{align*}
\end{theorem}

The rest of this paper is dedicated to constructing the proof of Theorem~\ref{theorem:k-dim-variance}. The standard approach involves a counting argument on the prime factorisations of $q$ and $r$, followed by a sieve on the coprimality condition. In the $k$-dimensional problem, one can simply adapt the $1$-dimensional overlap estimates proven in \cite{aistleitner2023metrictheoryofapproximations, koukoulopoulos2024sharpquantitativeversionduffinschaeffer, pollington1990kdimensionalduffinschaeffer}, as shown in Section~\ref{sec:overlap}. From here, we construct three bilinear estimates, allowing us to bound the variance contributions of the problematic pairs of integers $(q,r)$ with many prime factors. In the $1$-dimensional case, the approach of \cite{koukoulopoulos2020duffin, koukoulopoulos2024sharpquantitativeversionduffinschaeffer} is a technical iterative construction via GCD graphs. This approach was generalised in \cite{hauke2024provingduffinschaefferconjecturegcd, vazquez2024almostsharpquantitativeduffinschaeffergcd}, providing a framework for similar bounds on general multiplicative functions. In Section~\ref{sec:bilinear}, we show how this framework can be employed to construct equivalent bilinear bounds in all dimensions $k\geq 2$. From here, the variance estimate follows with an approach adapted from \cite{koukoulopoulos2024sharpquantitativeversionduffinschaeffer}, as shown in Section~\ref{sec:variance}.

% \todoMH[inline]{You might want to put in some short Acknowledgement section - you can thank me for discussions/suggesting this research (not that I particularly care, but it somehow gives you credibility); you can thank TU Graz/Institute of AZT for their hospitality etc. You can also mention that this was written as part of your master thesis if you want to.}
% \todoCOR[inline]{Yes, I planned to do this before submitting it :)}
\subsection*{Acknowledgements}
The author would like to thank Manuel Hauke for introducing them to the topic and for guidance and discussion throughout the work. The work was conducted during a research stay in TU Graz, as part of the author's master's thesis in NTNU; the author would thus also like to thank the Institute of Analysis and Number Theory in TU Graz for their hospitality.

\section{Overlap estimates}\label{sec:overlap}
For readability, given $t\geq 1$ and integers $q,r$, we define:
\begin{align*}
    L_t(q,r)\coloneq \sum_{\substack{p\mid\frac{qr}{(q,r)^2},\\p>t}}\frac{1}{p}, \text{ and } \omega_t(q,r)\coloneq \#\set{p\mid \frac{qr}{(q,r)^2}}{p\leq t}.
\end{align*}
Moreover, define:
\begin{align*}
    D(q,r)\coloneq \frac{\max(r\psi(q),q\psi(r)}{\gcd(q,r)}.
\end{align*}
Note that one can write:
\begin{align*}
        D(q,r) &= \lcm(q,r) \max\left(\frac{\psi(q)}{q}, \frac{\psi(r)}{r}\right),
    \end{align*}
    and thus if $D(q,r)<1/2$, we have:
    \begin{align*}
        \frac{1}{\lcm(q,r)} > 2\max\left(\frac{\psi(q)}{q}, \frac{\psi(r)}{r}\right) \geq \frac{\psi(q)}{q} + \frac{\psi(r)}{r}.
    \end{align*}
    We must then have $\A_q\cap \A_r=\varnothing$; otherwise there exists $a_i,b_i$ coprime to $q,r$ respectively with:
    \begin{align*}
        \left\lvert \frac{a_i}{q} - \frac{b_i}{r}\right\rvert = \frac{1}{\lcm(q,r)} \left\lvert \frac{a_ir-b_iq}{\gcd(q,r)} \right\rvert < \frac{1}{\lcm(q,r)},
    \end{align*}
    which is a contradiction.
    % \todoCOR[]{Unsure if we need the explanation, but it's nice, unexplained thoroughly elsewhere from my knowledge, and will certainly go in the thesis}
    As such, from here we will assume $D(q,r)\geq 1/2$. The variance estimate requires us to use two overlap estimates. The first comes from \cite{pollington1990kdimensionalduffinschaeffer}. 
\begin{lemma}\cite{pollington1990kdimensionalduffinschaeffer}\label{lem:pv-overlap}
    Let $q\neq r$ be natural numbers, and let $D=D(q,r)$. Then we have:
    \begin{align*}
        \lambda_k(\A_q^k\cap \A_r^k)\ll_k \ind{D\geq 1/2} \lambda_k(\A_q^k)\lambda_k(\A_r^k)\me^{kL_D(q,r)}.
    \end{align*}
\end{lemma}
The second is adapted from \cite{koukoulopoulos2024sharpquantitativeversionduffinschaeffer} for the $k$-dimensional case; we first state their estimate.
\begin{lemma}\cite[Lemma~5.1]{koukoulopoulos2024sharpquantitativeversionduffinschaeffer}
    Let $q\neq r$ be natural numbers, and let $D=D(q,r)$. Then, for any $t\geq 1$ we have:
    \begin{align*}
        \lambda(\A_q\cap \A_r)\leq \ind{D\geq 1/2} \lambda(\A_q)\lambda(\A_r)\me^{2L_t(q,r)}\left(1+O\left(\frac{2^{\omega_t(q,r)}\log(4D)}{D} \right) \right).
    \end{align*}
\end{lemma}

The estimate can then be adjusted to the $k$-dimensional case straightforwardly:
\begin{lemma}\label{lem:k-dim-overlap}
    Let $q\neq r$ be natural numbers, and let $D=D(q,r)$. Then, for any $t\geq 1$ we have:
    \begin{align*}
        \lambda_k(\A_q^k\cap \A_r^k)\leq \ind{D\geq 1/2} \lambda_k(\A_q^k)\lambda_k(\A_r^k)\me^{2kL_t(q,r)}\left(1+O\left(\frac{2^{k\omega_t(q,r)}\log(4D)}{D} \right) \right).
    \end{align*}
\end{lemma}
\begin{proof}
    We have:
    \begin{align*}
        \lambda_k(\A_q^k\cap \A_r^k) &\leq (\lambda(\A_q\cap\A_r))^k\\
        &\leq \ind{D\geq 1/2} \lambda_k(\A_q^k)\lambda_k(\A_r^k)\me^{2kL_D(q,r)}\left(1+O\left(\frac{2^{\omega_t(q,r)}\log(4D)}{D} \right) \right)^k.
    \end{align*}
    Regarding the error term, we have:
    \begin{align*}
        \left(1+O\left(\frac{2^{\omega_t(q,r)}\log(4D)}{D} \right) \right)^k &= 1+\sum_{i=1}^k O_i\left(\frac{2^{i\omega_t(q,r)}(\log(4D))^i}{D^i} \right).
    \end{align*}
    Recall that we assume $D\geq 1/2$. To see the dominant term, note that $0<\log(4D)/D< 1$ for all $D\geq 3$, and bounded
    % \todoCOR[]{The bound is easily quantified, I left it out but I can add it back if you like}
    % \todoMH[]{No that's more than enough here.}
    for all $1/2\leq D < 3$. Thus we can write:
    \begin{align*}
        \left(1+O\left(\frac{2^{\omega_t(q,r)}\log(4D)}{D} \right) \right)^k &= 1+O_k\left(\frac{2^{k\omega_t(q,r)}\log(4D)}{D} \right),
    \end{align*}
    giving the required result.
\end{proof}

% --------------------
\section{Bilinear bounds}\label{sec:bilinear}
The variance estimate follows from three bilinear bounds, established in \cite{koukoulopoulos2024sharpquantitativeversionduffinschaeffer} via GCD graphs. To generalise these bounds into higher dimensions, we instead employ the framework of \cite{hauke2024provingduffinschaefferconjecturegcd} and \cite{vazquez2024almostsharpquantitativeduffinschaeffergcd}. In preparation of this, we first introduce some notation.

Let $\map{\psi,\theta}{\N}{[0,\infty)}$ be finitely supported, and let $\map{f,g}{\N}{[0,\infty)}$ be multiplicative functions. For $v\in\N$, define:
\begin{align*}
    \mu^f_\psi(v)\coloneq \frac{f(v)\psi(v)}{v}.
\end{align*}
Given $V\subset \N$, define:
\begin{align*}
    \mu^f_\psi(V)\coloneq \sum_{v\in V}\mu^f_\psi(v) = \sum_{v\in V}\frac{f(v)\psi(v)}{v}.
\end{align*}
Given $\E\subset \N\times \N$, define:
\begin{align*}
    \mu^{f,g}_{\psi,\theta}(\E)\coloneq \sum_{(v,w)\in\E}\mu^f_\psi(v)\mu^g_\theta(w).
\end{align*}
Denoting now $V_\psi=\supp \psi$ and $W_\theta = \supp \theta$, for all $t\geq 1$ and $K\in \R$, define:
\begin{align*}
    \E^{t,K}_{\psi,\theta}\coloneq \set{(v,w)\in V_\psi\times W_\theta}{D_{\psi,\theta}(v,w)\leq 1, L_t(v,w)\geq K},
\end{align*}
where:
\begin{align*}
    D_{\psi,\theta}(v,w) \coloneq\frac{\max(w\psi(v),v\theta(w))}{(v,w)}.
\end{align*}
Similarly, define:
\begin{align*}
    \mathcal{F}^{t,K}_{\psi,\theta}\coloneq \set{(v,w)\in V_\psi\times W_\theta}{D_{\psi,\theta}(v,w)\leq 1, \omega_t(v,w)\geq K}.
\end{align*}
The first lemma below then provides a bound in the case where $L_t(v,w)$ is large; the second following provides a bound when $\omega_t(v,w)$ is large.
\begin{lemma}\cite[Theorem~1.7]{hauke2024provingduffinschaefferconjecturegcd}\label{lem:bilinear-e12}
    Let $\ep\in(0,2/5]$. Then there exists $p_0(\ep)>0$ such that the following holds. Let $\map{\psi,\theta}{\N}{[0,\infty)}$ be finitely supported, with $V_\psi=\supp \psi$, $W_\theta = \supp \theta$, and:
    \begin{align*}
        \mathcal{P}_{\psi,\theta}\coloneq \set{p}{\exists (v,w)\in V_\psi\times W_\theta \text{ such that } p\mid vw}.
    \end{align*}
    Let $P_{\psi,\theta}(\ep)\coloneq p_0(\ep)+\lvert \mathcal{P}_{\psi,\theta}\cap[1,p_0(\ep)]\rvert$. Let $\map{f,g}{\N}{[0,\infty)}$ be multiplicative functions such that:
    \begin{align*}
        (1\star f)(n)\leq n, \text{ and } (1\star g)(n)\leq n, \text{ for all }n\geq 1.
    \end{align*}
    Suppose that $\E \subset \E^{t,K}_{\psi,\theta}\cap (V\times W)$. Then for all $t\geq 1$ and $K\in \R$, we have:
    \begin{align*}
        \mu^{f,g}_{\psi,\theta}(\E) \leq 1000^{P_{\psi,\theta}(\ep)}(\mu^f_\psi(V)\mu^g_\theta(W)e^{-Kt})^{\tfrac{1}{2}+\ep}.
    \end{align*}
\end{lemma}
\begin{lemma}\cite[Theorem~1.7]{vazquez2024almostsharpquantitativeduffinschaeffergcd}\label{lem:bilinear-e3}
    Let $\ep\in(0,2/5]$ and $C>0$. Then there exists $p_0(\ep,C)>0$ such that the following holds. Let $\map{\psi,\theta}{\N}{[0,\infty)}$ be finitely supported, with $V_\psi=\supp \psi$, $W_\theta = \supp \theta$, and:
    \begin{align*}
        \mathcal{P}_{\psi,\theta}\coloneq \set{p}{\exists (v,w)\in V_\psi\times W_\theta \text{ such that } p\mid vw}.
    \end{align*}
    Let $P_{\psi,\theta}(\ep,C)\coloneq p_0(\ep)+\lvert \mathcal{P}_{\psi,\theta}\cap[1,p_0(\ep,C)]\rvert$. Let $\map{f,g}{\N}{[0,\infty)}$ be multiplicative functions such that:
    \begin{align*}
        (1\star f)(n)\leq n, \text{ and } (1\star g)(n)\leq n, \text{ for all }n\geq 1.
    \end{align*}
    Suppose that $\mathcal{F} \subset \mathcal{F}^{t,K}_{\psi,\theta}\cap (V\times W)$. Then for all $t\geq 1$ and $K\in \R$, we have:
    \begin{align*}
        \mu^{f,g}_{\psi,\theta}(\mathcal{F}) \leq (100e^C)^{P_{\psi,\theta}(\ep,C)}(\Log t)^{\tfrac{1}{2}(e^{40C}-1)}(\mu^f_\psi(V)\mu^g_\theta(W)e^{-CK})^{\tfrac{1}{2}+\ep}.
    \end{align*}
\end{lemma}

From here, we can construct the three key bilinear bounds, which can be seen as $k$-dimensional generalisations of Propositions 7.1-3 of \cite{koukoulopoulos2024sharpquantitativeversionduffinschaeffer}.
\begin{proposition}\label{proposition:overlap-sum-1}
    Fix $\ep\in (0,1)$, $k\geq 2$, let $\map{\psi}{\N}{[0,\infty)}$, and $y\geq 1$. Then:
    \begin{align*}
        \sum_{\substack{1\leq q,r\leq Q,\\D(q,r)\leq y}}\left(\frac{\varphi(q)\psi(q)}{q}\right)^k\left(\frac{\varphi(r)\psi(r)}{r}\right)^k\ll_{\ep,k} y^{1-\ep}\Psi_k(Q)^{1+\ep}.
    \end{align*}
\end{proposition}
\begin{proof}
    Let $f(n)=g(n)=(\varphi(n))^k/n^{k-1}$, $\Tilde{\psi}(q)=\ind{q\leq Q}\psi(q)^k/y$. Then $f$ is multiplicative, and:
    \begin{align*}
        (1\star f)(n)=\sum_{d\mid n}f(d)=\sum_{d\mid n}\varphi(d)\left(\frac{\varphi(d)}{d}\right)^{k-1} \leq \sum_{d\mid n}\varphi(d)=n.
    \end{align*}
    Moreover, we have:
    \begin{align*}
        D_{\Tilde{\psi},\Tilde{\psi}}(q,r)= \lcm(q,r)\max\left(\frac{\psi(q)^k}{yq},\frac{\psi(r)^k}{yr}\right)\leq \lcm(q,r)\max\left(\frac{\psi(q)}{yq},\frac{\psi(r)}{yr}\right)=\frac{D(q,r)}{y}\leq 1.
    \end{align*}
    Defining $\E=\set{(q,r)\in [1,Q]^2}{D(q,r)\leq y}$, we thus have $\E\subset \E^{0,0}_{\Tilde{\psi},\Tilde{\psi}}$, and applying Lemma~\ref{lem:bilinear-e12} gives:
    \begin{align*}
        \sum_{(q,r)\in\E}\left(\frac{\psi(q)\varphi(q)}{q}\right)^k\left(\frac{\psi(r)\varphi(r)}{r}\right)^k &= y^2\mu^{f,f}_{\Tilde{\psi},\Tilde{\psi}}(\E)\\
        &\ll_{\varepsilon,k} y^2\left(\frac{1}{y^2}\mu^f_{\Tilde{\psi}}(\{q\leq Q\})^2\right)^{\tfrac{1}{2}+\tfrac{\ep}{2}}\\
        &= y^{1-\ep}\Psi_k(Q)^{1+\ep}.\qedhere
    \end{align*}
\end{proof}
\begin{proposition}\label{proposition:overlap-sum-2}
    Fix $\ep\in (0,1)$, $k\geq 2$ and $C\geq 1$. Let $\map{\psi}{\N}{[0,\infty)}$, and $y,t,s\geq 1$. Then:
    \begin{align*}
        \sum_{\substack{1\leq q,r\leq Q,\\D(q,r)\leq y,\\L_t(q,r)\geq 1/s}}\left(\frac{\varphi(q)\psi(q)}{q}\right)^k\left(\frac{\varphi(r)\psi(r)}{r}\right)^k\ll_{\ep,k,C} \me^{-Ct/s} y^{1-\ep}\Psi_k(Q)^{1+\ep}.
    \end{align*}
    \begin{proof}
    Let $f(n)=g(n)=(\varphi(n))^k/n^{k-1}$, $\Tilde{\psi}(q)=\ind{q\leq Q}\psi(q)^k/y$. Then following the proof of Proposition~\ref{proposition:overlap-sum-1} and defining $\E=\set{(q,r)\in [1,Q]^2}{D(q,r)\leq y,\ L_t(q,r)\geq \frac{1}{s}}$, we have $\E\subset \E^{t,1/s}_{\Tilde{\psi},\Tilde{\psi}}$, and applying Lemma~\ref{lem:bilinear-e12} gives:
    \begin{align*}
        \sum_{(q,r)\in\E}\left(\frac{\psi(q)\varphi(q)}{q}\right)^k\left(\frac{\psi(r)\varphi(r)}{r}\right)^k &= y^2\mu^{f,f}_{\Tilde{\psi},\Tilde{\psi}}(\E)\\
        &\ll_{\varepsilon,k} y^2\left(\frac{1}{y^2}\mu^f_{\Tilde{\psi}}(\{q\leq Q\})^2e^{-t/s}\right)^{\tfrac{1}{2}+\tfrac{\ep}{2}}\\
        &\ll_{\ep,k,C} e^{-Ct/s}y^{1-\ep}\Psi_k(Q)^{1+\ep}.\qedhere
    \end{align*}
    \end{proof}
\end{proposition}
\begin{proposition}\label{proposition:overlap-sum-3}
    Fix $\ep\in (0,1)$, $k\geq 2$, $\kappa> 0$ and $C\geq 1$. Let $\map{\psi}{\N}{[0,\infty)}$, and $y,t\geq 1$. Then:
    \begin{align*}
        \sum_{\substack{1\leq q,r\leq Q,\\D(q,r)\leq y,\\\omega_t(q,r)\geq\kappa\log t}}\left(\frac{\varphi(q)\psi(q)}{q}\right)^k\left(\frac{\varphi(r)\psi(r)}{r}\right)^k\ll_{\ep, k, \kappa, C} t^{-C}y^{1-\ep}\Psi_k(Q)^{1+\ep}.
    \end{align*}
    \begin{proof}
        Let $f(n)=g(n)=(\varphi(n))^k/n^{k-1}$, $\Tilde{\psi}(q)=\ind{q\leq Q}\psi(q)^k/y$. Again following the proof of Proposition~\ref{proposition:overlap-sum-1}, now defining $\mathcal{F}=\set{(q,r)\in [1,Q]^2}{D(q,r)\leq y,\ \omega_t(q,r)\geq \kappa\log t}$, we have $\mathcal{F}\subset \mathcal{F}^{t,\kappa \log t}_{\Tilde{\psi},\Tilde{\psi}}$, and applying Lemma~\ref{lem:bilinear-e3} gives:
        \begin{align*}
            \sum_{(q,r)\in\mathcal{F}}\left(\frac{\psi(q)\varphi(q)}{q}\right)^k\left(\frac{\psi(r)\varphi(r)}{r}\right)^k &= y^2\mu^{f,f}_{\Tilde{\psi},\Tilde{\psi}}(\mathcal{F})\\
        &\ll_{\ep,k,C} y^2\left(\frac{1}{y^2}\mu^f_{\Tilde{\psi}}(\{q\leq Q\})^2e^{-C\kappa\log t}\right)^{\tfrac{1}{2}+\tfrac{\ep}{2}}\\
        &\ll_{\ep,k,\kappa,C} t^{-C}y^{1-\ep}\Psi_k(Q)^{1+\ep}.\qedhere
        \end{align*}
    \end{proof}
\end{proposition}

% --------------------
\section{Variance estimate}\label{sec:variance}

Finally, we prove the variance estimate, and thus conclude the proof of our main result. The proof follows \cite{koukoulopoulos2024sharpquantitativeversionduffinschaeffer} with minor modifications.
\begin{proof}[Proof of Theorem~\ref{theorem:k-dim-variance}]
    Note that we have:
    \begin{align*}
        \lint{\left[0,1\right]^k}{}{\left(S_k(\alpha, Q)-\Psi_k(Q)\right)^2}{\alpha}= \sum_{q,r\leq Q} \lambda_k(\A_q^k \cap \A_r^k) - \Psi_k(Q)^2,
    \end{align*}
    and so we will show instead that:
    \begin{align*}
        \sum_{q,r\leq Q} \lambda_k(\A_q^k \cap \A_r^k) \leq \Psi_k(Q)^2 + \Psi_k(Q) + O_{\ep,k}(\Psi_k(Q)^{1+\ep}).
    \end{align*}
    When $q=r$, we have simply $\sum_{q\leq Q}\lambda_k(\A_q^k) = \Psi_k(Q)$. Thus, consider the sum of terms $q\neq r$, and assume that $D=D(q,r)\geq 1/2$. We consider the contributions over the following partition:
    \begin{align*}
        \E^{(1)} &\coloneq\set{(q,r)\in \left[1,Q\right]^k}{L_{D^2}(q,r)\leq \frac{1}{D},\ \omega_{D^2}(q,r)\leq \frac{\ep}{4k}\log(2D)},\\
        \E^{(2)} &\coloneq\set{(q,r)\in \left[1,Q\right]^k}{L_{D^2}(q,r)> \frac{1}{D}},\\
        \E^{(3)} &\coloneq\set{(q,r)\in \left[1,Q\right]^k}{L_{D^2}(q,r)\leq \frac{1}{D},\ \omega_{D^2}(q,r)> \frac{\ep}{4k}\log(2D)}.
    \end{align*}
    First, consider the sum over $\E^{(1)}$. Then applying Lemma~\ref{lem:k-dim-overlap} with $t=D^2$ gives:
    \begin{align*}
        \lambda_k(\A_q^k\cap\A_r^k) &\leq \lambda_k(\A_q^k)\lambda_k(\A_r^k)\me^{2k/D}\left(1+O\left(\frac{2^{\ep\log(2D)/4}\log(4D)}{D}\right)\right)\\
        &\leq \lambda_k(\A_q^k)\lambda_k(\A_r^k)\left(1+O_k(D^{-1+\ep/2})\right)
    \end{align*}
    Therefore:
    \begin{align*}
        \sum_{(q,r)\in\E^{(1)}}\lambda_k(\A_q^k\cap\A_r^k) &\leq \sum_{(q,r)\in\E^{(1)}}\lambda_k(\A_q^k)\lambda_k(\A_r^k) + O_k\left(\sum_{(q,r)\in\E^{(1)}}\lambda_k(\A_q^k)\lambda_k(\A_r^k)D^{-1+\ep/2}\right)\\
        &\leq \Psi_k(Q)^2+O_k\left(\sum_{(q,r)\in\E^{(1)}}\lambda_k(\A_q^k)\lambda_k(\A_r^k)D^{-1+\ep/2}\right).
    \end{align*}
    Focusing on the error term, define $R\coloneq\sum_{(q,r)\in\E^{(1)}}\lambda_k(\A_q^k)\lambda_k(\A_r^k)D^{-1+\ep/2}$, and let $I$ be the smallest non-negative integer such that $2^I\geq \Psi_k(Q)$. We then split $R$ via dyadic partition on $D$, writing $R_i$ as the part of the sum with $D\in [2^{i-1},2^i)$, and $R'$ the part with $D\geq 2^I$. Then:
    \begin{align*}
        R'&\leq \Psi_k(Q)^{-1+\ep/2}\sum_{(q,r)\in\E^{(1)}}\lambda_k(\A_q^k)\lambda_k(\A_r^k) \leq \Psi_k(Q)^{1+\ep/2},
    \end{align*}
    and for $i\in \{0,1,\dots,I\}$ we have by Proposition~\ref{proposition:overlap-sum-1} with $y=2^i$ that:
    \begin{align*}
        R_i &\leq (2^{i-1})^{-1+\ep/2}\sum_{(q,r)\in\E^{(1)}}\lambda_k(\A_q^k)\lambda_k(\A_r^k)\\
        &\ll_{\ep,k} (2^i)^{-1+\ep/2}(2^i)^{1-\ep}\Psi_k(Q)^{1+\ep}\\
        &= 2^{-i\epsilon/2}\Psi_k(Q)^{1+\ep}.
    \end{align*}
    We thus have $R \ll_{\ep,k}\Psi_k(Q)^{1+\ep}$ and so:
    \begin{align*}
        \sum_{(q,r)\in\E^{(1)}}\lambda_k(\A_q^k\cap\A_r^k) \leq \Psi_k(Q)^k + O_{\ep,k}(\Psi_k(Q)^{1+\ep}).
    \end{align*}
    Considering $\E^{(2)}$, we use Lemma~\ref{lem:pv-overlap} to say:
    \begin{align*}
        \sum_{(q,r)\in\E^{(2)}}\lambda_k(\A_q^k\cap\A_r^k) &\ll_k \sum_{(q,r)\in\E^{(2)}}\lambda_k(\A_q^k)\lambda_k(\A_r^k)\me^{kL_D(q,r)}.
    \end{align*}
    If $D\in [2^{i-1},2^i)$, we have $L_{4^{i-1}}(q,r)\geq L_{D^2}(q,r)>1/D\geq 2^{-i}$. We then write:
    \begin{align*}
        \E^{(2)}=\bigcup_{i=0}^\infty \bigcup_{j=i}^\infty \E^{(2)}_{i,j},
    \end{align*}
    where:
    \begin{align*}
        \E^{(2)}_{i,i}&\coloneq\set{(q,r)\in \E^{(2)}}{D(q,r)\in[2^{i-1},2^i),\ L_{4^{i-1}}(q,r)\leq 1},\text{ and}\\
        \E^{(2)}_{i,j}&\coloneq\set{(q,r)\in \E^{(2)}}{D(q,r)\in[2^{i-1},2^i),\ L_{4^{j-1}}(q,r)\leq 1<L_{4^{j-2}}(q,r)},\text{ for }j>i.
    \end{align*}
    Then, if $(q,r)\in\E^{(2)}_{i,j}$, we have:
    \begin{align*}
        L_D(q,r)\leq L_{2^{i-1}}(q,r)\leq \sum_{2^{i-1}< p\leq 4^{j-1}}\frac{1}{p}+L_{4^{j-1}}(q,r) \leq \log\frac{j+1}{i+1} + O(1),
    \end{align*}
    and so:
    \begin{align*}
        \sum_{(q,r)\in\E^{(2)}}\lambda_k(\A_q^k\cap\A_r^k) \ll_k \sum_{j\geq i\geq 0}\left(\frac{j+1}{i+1}\right)^k\sum_{(q,r)\in\E^{(2)}_{i,j}}\lambda_k(\A_q^k)\lambda_k(\A_r^k).
    \end{align*}
    Then for $i\geq 0$, Proposition~\ref{proposition:overlap-sum-2} with $y=2^i$, $t=\max\{1, 4^{i-1}\}$, $s=2^i$ and $C=8$ gives:
    \begin{align*}
        \sum_{(q,r)\in\E^{(2)}_{i,i}}\lambda_k(\A_q^k)\lambda_k(\A_r^k) &\ll_{\ep,k} \exp(-2^i)\Psi_k(Q)^{1+\ep}.
    \end{align*}
    Similarly, for $j>i$, choosing $y=2^i$, $t=\max\{1, 4^{j-2}\}$, $s=1$ and $C=17$ gives:
    \begin{align*}
        \sum_{(q,r)\in\E^{(2)}_{i,j}}\lambda_k(\A_q^k)\lambda_k(\A_r^k) \ll_{\ep,k} \exp(-4^j)\Psi_k(Q)^{1+\ep}.
    \end{align*}
    Together, this gives:
    \begin{align*}
        \sum_{(q,r)\in\E^{(2)}}\lambda_k(\A_q^k\cap\A_r^k) &\ll_{\ep,k} \Psi_k(Q)^{1+\ep}\left(\sum_{i\geq 0}\exp(-2^i) +\sum_{j>i}\left(\frac{j+1}{i+1}\right)^k\exp(-4^j)\right)\\
        &\ll_{\ep,k} \Psi_k(Q)^{1+\ep}.
    \end{align*}
    Finally, considering $\E^{(3)}$, we have that:
    \begin{align*}
        L_D(q,r)\leq \sum_{D<p\leq D^2}\frac{1}{p} + L_{D^2}(q,r) \ll \log\log D^2-\log\log D + \frac{1}{D} \ll 1.
    \end{align*}
    We thus have by Lemma~\ref{lem:pv-overlap} that:
    \begin{align*}
        \sum_{(q,r)\in\E^{(3)}}\lambda_k(\A_q^k\cap\A_r^k) &\ll_k \sum_{(q,r)\in\E^{(3)}}\lambda_k(\A_q^k)\lambda_k(\A_r^k).
    \end{align*}
    If $D\in [2^{i-1},2^i)$, we have:
    \begin{align*}
        \omega_{4^i}(q,r)\geq \omega_{D^2}(q,r)&> \frac{\ep}{4k}\log(2D) \geq \frac{\ep}{8k}\log(4^i).
    \end{align*}
    % \todoMH[]{What is the purposse of the last inequality?}
    We can thus write:
    \begin{align*}
        \sum_{(q,r)\in\E^{(3)}}\lambda_k(\A_q^k\cap\A_r^k) &\ll_k \sum_{i\geq 0}\sum_{(q,r)\in\E^{(3)}_i}\lambda_k(\A_q^k)\lambda_k(\A_r^k),
    \end{align*}
    where, for each $i\geq 0$:
    \begin{align*}
        \E^{(3)}_{i}&\coloneq\set{(q,r)\in \E^{(3)}}{D(q,r)\in[2^{i-1},2^i),\ \omega_{4^i}(q,r)> \frac{\ep}{8k}\log(4^i)}.
    \end{align*}
    Applying Proposition~\ref{proposition:overlap-sum-3} with $y=2^i$, $t=4^i$, $\kappa=\ep/8k$ and $C=2$ gives:
    \begin{align*}
        \sum_{(q,r)\in\E^{(3)}_i}\lambda_k(\A_q^k)\lambda_k(\A_r^k) &\ll_{\ep,k} (4^i)^{-2}(2^i)^{1-\ep}\Psi_k(Q)^{1-\ep} \ll_\ep 2^{-i}\Psi_k(Q)^{1+\ep},
    \end{align*}
    and so:
    \begin{align*}
        \sum_{(q,r)\in\E^{(3)}}\lambda_k(\A_q^k\cap\A_r^k) &\ll_{\ep,k} \sum_{i\geq 0}2^{-i}\Psi_k(Q)^{1+\ep} \ll_\ep\Psi_k(Q)^{1+\ep}.
    \end{align*}
    Putting everything together:
    \begin{align*}
        \sum_{q,r\leq Q}\lambda_k(\A_q^k\cap\A_r^k) &= \sum_{q\leq Q}\lambda_k(\A_q^k) + \sum_{i=1}^3\sum_{(q,r)\in\E^{(i)}}\lambda_k(\A_q^k\cap\A_r^k)\\
        &\leq \Psi_k(Q) + \Psi_k(Q)^2 + O_{\ep,k}(\Psi_k(Q)^{1+\ep}),
    \end{align*}
    as required.
\end{proof}

% --------------------
\printbibliography

\end{document}